\documentclass[a4paper, hidelinks, 11pt]{article}
\usepackage[utf8]{inputenc}
\usepackage[T1]{fontenc}
\usepackage{graphicx} % Include figure files
\usepackage{amsmath, relsize} 
\usepackage{todonotes} 
\usepackage{amsthm} 
\usepackage{amsfonts}
\usepackage{amssymb}
\usepackage{url}
\usepackage{mathrsfs}
\usepackage{amsmath}
\usepackage{amssymb}
\usepackage{amsthm}
\usepackage{bbm}
\usepackage{comment}
\usepackage{caption}
\usepackage{subcaption}
\usepackage{tocloft}

\usepackage{enumerate}
\usepackage{setspace}
\usepackage{graphicx,color,hyperref}
\usepackage[margin=1.2in]{geometry}
\parskip \medskipamount
\newtheorem{theorem}{Theorem}[section]
\newtheorem{proposition}[theorem]{Proposition}

\theoremstyle{definition}

\newtheorem{remark}[theorem]{Remark}

\newcommand{\V}     {\mathbb{V}} 
\newcommand{\C}     {\mathbb{C}} 
\newcommand{\G}     {\mathbb{G}} 
\newcommand{\R}     {\mathbb{R}} 
\newcommand{\Z}     {\mathbb{Z}} 
\newcommand{\N}     {\mathbb{N}} 
\renewcommand{\P}   {\mathbb{P}} 
 
\newcommand{\E}     {\mathbb{E}}

\newcommand{\Dcal}   {{\mathcal D }}

\newcommand{\Gcal}   {{\mathcal G }}

\newcommand{\Mcal}   {{\mathcal M }}

\usepackage{authblk}
\title{Decay of correlations in the  monomer-dimer model}
\author{Alexandra Quitmann\thanks{University of Rome, La Sapienza, Rome, Italy. Email: alexandra.quitmann@uniroma1.it }}
\numberwithin{equation}{section}

\numberwithin{equation}{section}

\begin{document}

\date{
    \today
}

\maketitle

\begin{abstract}
We consider the monomer-dimer model, whose realisations are spanning sub-graphs of a given graph such that
every vertex has degree zero or one.
The measure depends on a parameter,  the monomer activity,  
which rewards the total number of monomers.
We consider general correlation functions including monomer-monomer correlations and dimer-dimer covariances. We show that these correlations decay exponentially
fast with the distance if the monomer activity is strictly positive. 
Our result improves a previous upper bound from van den Berg and is of interest due to its relation to transverse spin-spin correlations in classical spin systems.
Our proof is based on the cluster expansion technique.% and  the Lee and Yang theorem. 
\end{abstract}

\section{Introduction}
This note considers the \textit{monomer-dimer model} \cite{HeilmannLieb},
whose realisations are spanning sub-graphs of a given graph such that
every vertex has degree zero or one.  Vertices with degree zero 
are referred to as \textit{monomers} and pairs of vertices connected by an edge
are referred to as \textit{dimers}.  
The measure depends on a parameter,  the monomer activity $\rho \geq 0$,  
which controls the total number of monomers. In case of zero monomer
activity no monomers are present and we obtain the classical \textit{dimer model}.
%The model attracts interests...Hier jetzt Referenzen nennen

By superimposing two realisations of the monomer-dimer model we obtain a configuration of the \textit{double monomer-dimer model}. 
This model can be viewed as a random walk loop soup whose configurations are collections of self-avoiding and mutually self-avoiding paths which might be closed or open, see e.g. Figure \ref{Fig:monomerdimer}. If the monomer activity is zero, all paths are closed and the double monomer-dimer model reduces to the double dimer model.

The (double) monomer-dimer model shares an intriguing similarity with the Spin $O(N)$ model, namely they both have a probabilistic reformulation as a particular random path model \cite{LeesTaggi}. In this representation, the external magnetic field of the Spin $O(N)$ model plays the same role as the monomer activity in the monomer-dimer model and the two models only differ in the weight that is assigned to the number of visits of (open and closed) paths at the vertices. 
%, which is one of the most important statistical mechanics models. Its configurations are vectors taking values in the $N-1$ dimensional unit sphere and the measure depends on a parameter, which is called the external magnetic field. 
%This relation will be made more precise later in this note.
 The qualitative behaviour of the random path model, however, is expected to not depend on the choice of such weight function. 
 %In particular, through this reformulation, questions posed in one of the two models, are closely related to the other one.
%For non-zero monomer activity or external magnetic field 
%In particular, both, monomer-monomer and spin-spin correlations can be expressed in terms of these random paths. 
%In \cite{LeesTaggi} this reformulation was applied to extend the proof of exponential decay of the truncated spin-spin correlations to any integer value of $N$. 

In this note we study the rate of decay of \textit{monomer-monomer correlations} when the monomer activity is non-zero.
Through the random path representation, this question is closely related to an open problem in the Spin $O(N)$ model. Here, it is known that connected spin-spin correlations decay exponentially fast with the distance between the vertices when the external magnetic field is non-zero. The constant of decay in the exponent is known to be $\Omega(h)$ as $h \to 0$ for $N=1,2,3$ \cite{Frohlich} and to be $\Omega(h^2)$ for any $N \in \N$ \cite{LeesTaggi}.
% In dem paper werden nicht die truncated correlations betrachtet!
 It is however conjectured that the constant decays as $\Omega(\sqrt{h})$ when $h \to 0$ for any integer value of $N$. The same behaviour is expected to occur in the monomer-dimer model. 

Our main result shows that for any strictly positive value of the monomer activity $\rho$, monomer-monomer correlations decay exponentially fast with the graph distance between the vertices.
% If the monomer activity is small enough, %, but strictly positive, 
%we further show that rate of decay is at least linear in the monomer activity. 
For large enough values of $\rho$, this result is derived using a cluster expansion. Applying similar analytic tools as in \cite{Frohlich} we can then extend this result to small values of $\rho$ and further show that the constant of decay is of order at least $\Omega(\rho)$ as $\rho \to 0$.
%It does not depend on the geometry of the underlying lattice and thus holds in any dimension. 
%For small values of $\rho$ we then apply similar analytic tools as in \cite{Frohlich} and we show that the constant of decay is at least linear in the monomer activity if the monomer activity is small enough. 
%Via the random path representation our problem is closely related to the rate of decay of spin-spin correlations in presence of an external magnetic field.
%Our arguments are based on \cite{Frohlich}, where a similar behaviour was shown for truncated spin-spin correlations.
%Nonetheless, for both models, it is still an open conjecture that the constant of decay is of order $O(\sqrt{\rho})$ in the limit as $\rho \to 0$. 
%$h$. Based on the Lee and Yang theorem, it has been shown in \cite{Frohlich} that for any $N \in \{1,2,3\}$ these correlations admit exponential decay with a constant that decays at least linear in the magnetic field in the limit as $h \to 0$. 

%Via the random path representation, monomer-monomer correlations in the monomer-dimer model are closely related to truncated spin-spin correlations in the Spin $O(N)$ model. For the latter model, it is conjectured that the constant of decay is of order  $O(\sqrt{h})$ in the limit as $h \to 0$ and the same behaviour is expected to occur for our monomer correlations. 

It should be emphasized that our result only holds for non-zero values of the monomer activity and the behaviour of the dimer model, i.e., the monomer-dimer model at zero monomer activity, is strongly different. In dimension $d=2$ the monomer-monomer correlation admits polynomial decay with the distance between the monomers \cite{Dubedat,Fisher}, while in any dimension $d \geq 3$ it exhibits long-range order \cite{Taggi}. 

Our result also holds for more general correlation functions including the \textit{dimer-dimer covariance} as special case. It is known that this covariance decays exponentially in the distance between the edges with a constant of order $\Omega(\rho^2)$ in the limit as $\rho \to 0$. More precisely, in \cite{vandenBergSteif} it is shown that the dimer-dimer covariance is upper bounded by the probability of observing a path in the  double monomer-dimer model that connects these two edges. The exponential decay of such probability then follows from \cite{vandenBerg} based on an argument with disagreement percolation. We remark that for non-zero monomer activity, the connection probability behaves differently, %In contrast, if the monomer activity is zero, such probability also tends to zero in two dimensions \cite{Kenyon}, however,
namely it stays uniformly positive in any dimension $d \geq 3$ \cite{QuitmannTaggi}. 
In this note, we improve the existing bound due to \cite{vandenBerg,vandenBergSteif} by showing that the constant decays as $\Omega(\rho)$ in the limit $\rho \to 0$.

\begin{figure}
\begin{center}
\includegraphics[scale=0.6]{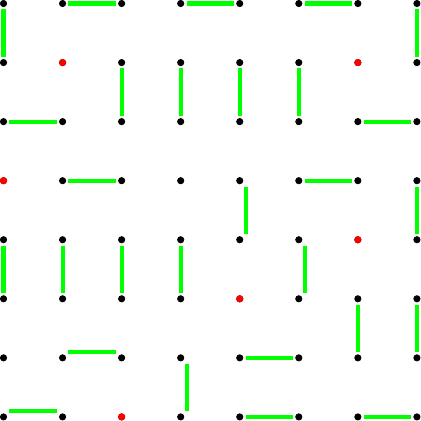}
\hspace{0.7cm}
\includegraphics[scale=0.6]{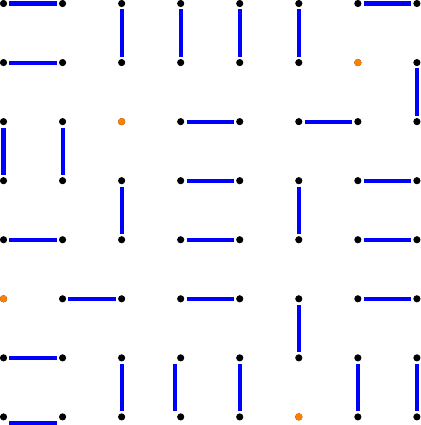}
\hspace{0.7cm}
\includegraphics[scale=0.6]{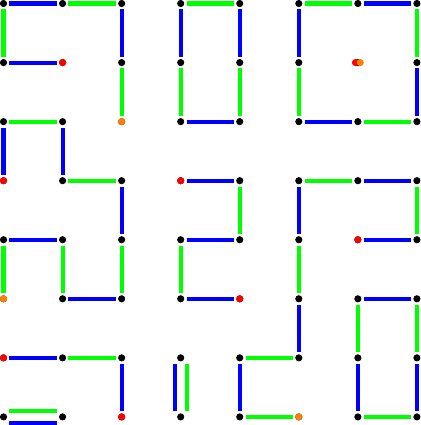}
\caption{The first two figures show monomer-dimer configurations $\omega, \omega^\prime \in \Omega_K$, where $K \subset \Z^2$. The third figure shows their superposition resulting in a collection of open and closed paths.}
\label{Fig:monomerdimer}
\end{center}
\end{figure}

It is further interesting to compare the double monomer-dimer model with the \textit{monomer double-dimer model} \cite{QuitmannTaggi}. 
In the latter model, the configurations are also  superpositions of two  independently sampled monomer-dimer configurations, however, conditioned on having the same set of monomers. 
Consequently, the paths in the double monomer-dimer model might be open, while in the monomer double-dimer model all paths are closed. In the discussion above we have seen that the double monomer-dimer model behaves drastically different if the monomer activity changes from small, but strictly positive values to zero. This, however, is not the case in the monomer double-dimer model, i.e., in the system where all the paths are closed \cite{BetzTaggi, QuitmannTaggiQuattropaniForien, QuitmannTaggi1, QuitmannTaggi}. In particular, exponential decay of the connection probabilities only occurs for large enough values of the monomer activity.

\section{Model and main result}
Consider a finite undirected graph $G = (V,E)$. 
A dimer configuration (or perfect matching) of $G$ is a subset $d \subset E$ of edges such that every vertex $v \in V$ is an element of precisely one edge. 
We let $D_G$ be the set of all dimer configurations in $G$. 
Given a set $A \subset V$, we let $G_A$ be the subgraph of $G$ with vertex set
$V \setminus A$ and with edge-set consisting of all the edges  in $E$ which 
do not touch any vertex in $A$.
We let $D_G(A)$ be the set of dimer configurations in $G_A$. 

In this note, we concentrate on the $d$-dimensional cubic lattice. We denote by $\G=(\V,\E)$ the graph with vertex set $\V=\Z^d$ and with edge set $\E=\{\{x,y\}  \, : x,y \in \Z^d, \, d(x,y)=1\}$, where $d(x,y)$ denotes the graph distance between $x$ and $y$. 
We denote by $G_K=(K,E_K)$ the graph with vertex set $K \subset \Z^d$ and with edge set $E_K := \{\{x,y\} \in \E : \, x,y \in K \} \subset \E$.

%\subsection{The monomer-dimer model.}
Given $K \subset \Z^d$, the configuration space of the monomer-dimer model in $G_K$ is denoted by $\Omega_K$ and it corresponds to the set of tuples $\omega = (M, d)$ such that $M \subset K$ and $d \in D_{G_K}(M) $. 
We refer to the first and second element of the tuple $\omega$ as a set of monomers and a set of dimers, respectively. 
We let $\Mcal:\Omega_K \to K$ and $\Dcal: \Omega_K \to E_K$ be the random variables defined by $\mathcal{M}(\omega) := M$ and $\Dcal(\omega):=d$ for each $\omega = (M, d) \in \Omega_K$.

We define a probability measure on $\Omega_K$,
\begin{equation} \label{eq:defmonomerdimermeasure}
\forall \omega \in \Omega_K \qquad \P_{K,\rho}(\omega) := \frac{\rho^{|\mathcal{M}(\omega)|}}{\Z_{K,\rho}},
\end{equation}
where $\rho \geq 0$ is the parameter of the model (monomer density) and $\Z_{K,\rho}$ is the normalizing constant (partition function). 
%We set
%$$
%E:= \bigcup_{L \in \N} E_L
%$$

We are interested in correlations between sets of monomers. %For any $\rho \geq 0$ and any $A \subset K$ we set
%$$
%C_{K,\rho}(A):= \Z_{K \setminus A, \rho}.
%$$
%In other words, $\textcolor{magenta}{\rho^{|A|}} \, C_{K,\rho}(A)$ corresponds to the weight of all monomer-dimer configurations in $G_K$ with fixed monomers at all vertices in $A$.  
For any $A,B \subset K$ with $A \cap B=\emptyset$, we introduce the correlation function
$$
U_{K,\rho}(A,B):= 
%\P_{K,\rho}(A \cup B \subset \Mcal) - \P_{K,\rho}(A \subset \Mcal) \, \P_{K,\rho}(B \subset \Mcal).
 \frac{\Z_{K \setminus A \cup B,\rho}}{\Z_{K,\rho}}- \frac{\Z_{K \setminus A,\rho}}{\Z_{K,\rho}} \, \frac{\Z_{K \setminus B,\rho}}{\Z_{K,\rho}}.
$$
Note that,
$$
\rho^{|A|+|B|} \; U_{K,\rho}(A,B) =  \P_{K,\rho}(A \cup B \subset \Mcal) - \P_{K,\rho}(A \subset \Mcal) \, \P_{K,\rho}(B \subset \Mcal).
$$
We further set
$$
%C_\rho(A):= \lim_{K \uparrow \Z^d} \frac{\Z_{K \setminus A,\rho}}{\Z_{K,\rho}} \text{ and } 
U_\rho(A,B):= \lim_{K \uparrow \Z^d} U_{K,\rho}(A,B),
%= C_\rho(A \cup B)-C_\rho(A) \, C_\rho(B),
$$
where the limit is in the sense of van Hove. Its existence follows from \cite[Theorem 10]{Gruber} since our monomer-dimer model is a special case of the polymer systems studied in \cite{Gruber}.
If $A,B \in \E$, then the correlation function reduces to the dimer-dimer covariance, namely
$$
U_{K,\rho}(A,B) = \P_{K,\rho}\big(A,B\in \Dcal\big) - \P_{K,\rho}\big(A \in \Dcal\big) \, \P_{K,\rho}\big(B \in \Dcal\big) .
$$

\paragraph{Monomer correlations, paths and $O(N)$ spin systems.} We now briefly explain the relation between monomer-monomer and spin-spin correlations. Consider the Spin $O(N)$ model with $N \in \N$ at inverse temperature $\beta \geq 0$ and external magnetic field $h \geq 0$. 
In \cite[Proposition 2.3]{LeesTaggi} it is shown that the spin-spin correlation at $x,y \in K \subset \V$ is identical to the two-point function $\G_{G_K,N,\beta,h}(x,y)$ that is defined in a particular model of random paths. The measure of this model depends on a function  $\mathcal{U}:\N_0 \to \R_{\geq 0}$ which controls the number of visits of paths at the vertices. If we consider a different choice of the function $\mathcal{U}$, namely if we set $\tilde{\mathcal{U}}(r):=1$ for any $r \in \N_0$, then we have that for $N=2$, $h=\rho$ and any $\beta \geq 0$, 
$$
\frac{\Z_{K \setminus \{x,y\},\rho}}{\Z_{K,\rho}}= \tilde{\G}_{G_K,N,\beta,h}(\{x,y\}),
$$
where $\tilde{\G}_{G_K,N,\beta,h}(x,y)$ is defined as the function $\G_{G_K,N,\beta,h}(x,y)$, but with the choice of $\tilde{\mathcal{U}}$ instead of $\mathcal{U}$.

We now present our main theorem. It states that correlation functions between sets of monomers decay exponentially fast with the distance between their vertices. %Furthermore, the constant in the exponent decays at least linear in the monomer density if $\rho$ is small. 
For any $A, B \subset \V$ we set 
$$
d(A,B):= \min_{\substack{(u,v): \, u \in A, \, v \in B}} d(u,v).
$$
 
\begin{theorem} \label{theorem:maintheorem} For any $d \geq 1$, $\rho>0$ and non-empty sets $A, B \subset \V$ with $A \cap B=\emptyset$, there exist $c=c(d,\rho,|A|,|B|), \, c^\prime=c^\prime(d,\rho,|A|,|B|) \in (0,\infty) $ such that 
\begin{equation} \label{eq:maintheorem}
\big|U_\rho(A,B)\big| \leq  c^\prime \, e^{-c \,  d(A,B) },
\end{equation}
where $c= \tilde{c}(d,|A|,|B|) \, \rho$ if $\rho$ is sufficiently small. 
%and $\tilde{c} \geq \frac{2}{\ln(2(a+1)) \, (a+1)}$ with $a=e \, \sqrt{e \, \textcolor{blue}{\min\{|A|,|B|\}} \, (4d-1)}$. 
%Further, $c^\prime \leq \max\{\rho^{-(|A|+|B|)}, \rho^{-2(|A|+|B|)} \} \, c^{\prime \prime}$ for some $c^{\prime \prime} \in (0,\infty)$ that does not depend on $\rho$.
\end{theorem}

\begin{remark}Concerning the decay of dimer-dimer covariances, our theorem above improves the result of \cite{vandenBerg, vandenBergSteif} which states that $c=\Omega(\rho^2)$ in the limit as $\rho \to 0$. We remark, however, that it is still an open problem to show that $c =\Omega(\sqrt{\rho})$. 
\end{remark}
%The same rate of decay is expected for spin-spin correlations in the Spin $O(N)$ model when the external magnetic field approaches zero.
\begin{remark}
The constant $c^\prime$ and a lower bound on $\tilde{c}$ can be written explicitly and follow directly from the proof of Theorem \ref{theorem:maintheorem}. 
\end{remark}

\section{Proof of Theorem \ref{theorem:maintheorem}}
In this section we present the proof of Theorem \ref{theorem:maintheorem}. 
We will first prove exponential decay of monomer correlations in the regime of sufficiently large $\rho$ using a cluster expansion. Exponential decay for small $\rho$ will then follow by applying an analytic theorem, see Theorem \ref{theorem:Guerra} below. 

\begin{proposition} \label{prop:exponentialdecaycluster}
For any $d \geq 1$, any $K \subset \Z^d$, any non-empty $A,B \subset \V$ with $A \cap B =\emptyset$ and any $\rho \geq \sqrt{e \, \min\{|A|,|B|\} \, (4d-1)} \, e$, it holds that 
\begin{equation} \label{eq: exponentialdecaycluster}
\big| \rho^{|A|+|B|} \; U_{K,\rho}(A,B) \big| \leq  e^{-2 \, d(A,B) \,+e^{-2}}.
\end{equation}
%In particular, for any $e,e^\prime$, 
%\begin{equation}
%\text{Cov}_{L,\rho}(e,e^\prime) \leq c_1 \, e^{-c_2 \, d(e,f)}.
%\end{equation}
\end{proposition}
\begin{proof}
Fix $d \geq 1$, $K \subset \Z^d$ and two non-empty sets $A,B \subset \V$ with $A \cap B = \emptyset$. Without loss of generality, assume that $|A| \leq |B|.$ Let $\rho \geq \sqrt{e \, |A| \, (4d-1)} \, e$.
To begin, we rewrite the partition function $\Z_{K, \rho}$ using a cluster expansion. First, we note that
\[
\begin{aligned}
\Z_{K, \rho} 
% =    \sum\limits_{(M, d ) \in {\Omega}} \rho^{|M|} 
& = \sum_{n=0}^{|K|/2} \sum_{\substack{(M, d ) \in {\Omega_K} : \\ |d|=n}} \rho^{|K|-2|d|}  \\
& = \rho^{|K|} \bigg(1 + \sum_{n \geq 1} \frac{\rho^{-2n}}{n!} \sum_{(\gamma_1,\dots,\gamma_n) \in E_K^n}  \prod_{1 \leq i < j \leq n } \big(1+ \zeta(\gamma_i,\gamma_j)\big)\bigg),
\end{aligned}
\]
% We have hard-core interactions $\delta(\gamma,\gamma^\prime)= \mathbbm{1}_{\{\gamma \cap \gamma^\prime = \emptyset \}} \in \{0,1\}$. Thus, the term in the brackets is of the form pf (5.4) in FV.
where for $\gamma,\gamma^\prime \in E_K$,  $\zeta(\gamma,\gamma^\prime):= \mathbbm{1}_{\{\gamma \cap \gamma^\prime = \emptyset\}} -1$.
We denote by $\mathcal{G}_n$ the set of all (unoriented) connected graphs with vertex set $\mathcal{V}_n=\{1,\dots,n\}$. We introduce the Ursell functions $\varphi$ on finite ordered sequences $(\gamma_1,\dots,\gamma_m) \in E_K^m$, which are defined by 
\begin{equation*} 
\varphi(\gamma_1,\dots,\gamma_m) := 
\begin{cases} 
1 & \text{ if } m=1, \\
\frac{1}{m!} \sum\limits_{G \in \Gcal_m} \prod\limits_{\{i,j\} \in G} \zeta(\gamma_i,\gamma_j) & \text{ if } m \geq 2,
\end{cases}
\end{equation*}
where the product is over all edges in $G$. 
For any $\gamma^* \in E_K$, using that $\rho \geq \sqrt{e \,  (4d-1)}$, it holds that 
\begin{equation*}
\sum_{\gamma \in E_K} \rho^{-2} e \, |\zeta(\gamma,\gamma^*)| \leq (4d-1) \, \rho^{-2} e \leq 1.
\end{equation*}
% Für jedes L ist $|\Gamma| < \infty$. Ist es ein Problem, wenn $L \to \infty$?
By cluster expansion \cite[Theorem 5.4 %and Theorem 5.8
]{FriedliVelenik} it thus holds that
\begin{equation} \label{eq:clusterexpansionZ}
\Z_{K, \rho}  = \rho^{|K|} \, \exp\bigg( \sum\limits_{m \geq 1} \sum\limits_{(\gamma_1,\dots,\gamma_m) \in E_K^m} \varphi(\gamma_1,\dots,\gamma_m) \, \rho^{-2m} \bigg),
\end{equation}
where combined sum and integrals converge absolutely. Furthermore, for any $\gamma_1 \in E_K$, we have that
\begin{equation} \label{eq:upperbound}
1 + \sum\limits_{n \geq 2} n \sum\limits_{(\gamma_2,\dots,\gamma_n) \in E_K^{n-1}} |\varphi(\gamma_1,\gamma_2, \dots, \gamma_n)| \, \rho^{-2(n-1)} \leq e.
\end{equation}

Using a similar cluster expansion as above, we further obtain that for any $K^\prime \subset K$, 
\begin{equation} \label{eq:clusterexpansionG}
\Z_{K\setminus K^\prime, \rho}
 = \rho^{|K|-| K^\prime|} \,  \exp\bigg( \sum\limits_{m \geq 1} \sum\limits_{(\gamma_1,\dots,\gamma_m) \in E_{K \setminus  K^\prime}^m} \varphi(\gamma_1,\dots,\gamma_m) \,\rho^{-2m} \bigg).
\end{equation}

For $m \in \N$ and $ K^\prime \subset K$, let $C_{K^\prime}^m$ denote the set of ordered sequences $\boldsymbol{\gamma}=(\gamma_1,\dots,\gamma_m) \in E_K^m$ such that there exists $i \in [m]$ and $x \in  K^\prime$ such that $x$ is an endpoint of $\gamma_i$.  

%Fix now two disjoint subsets $A, B \subset K$. 
From \eqref{eq:clusterexpansionZ} and \eqref{eq:clusterexpansionG} we then deduce that, 
\begin{equation} \label{eq:twopoint}
\begin{aligned}
& \P_{K,\rho}(A \cup B \subset \Mcal) \\ 
%& = \rho^{|A|+|B|} \frac{\Z_{K \setminus A \cup B, \rho}}{\Z_{K,\rho}} \\
&  = \exp\bigg(- \sum\limits_{m \geq 1} \sum\limits_{(\gamma_1,\dots,\gamma_m) \in C_{A}^m} \varphi(\gamma_1,\dots,\gamma_m) \, \rho^{-2m} \bigg) \\
& \qquad \times \exp\bigg(- \sum\limits_{m \geq 1} \sum\limits_{(\gamma_1,\dots,\gamma_m) \in C_{B}^m} \varphi(\gamma_1,\dots,\gamma_m) \, \rho^{-2m} \bigg) \\
& \qquad \qquad  \times \exp\bigg(\sum\limits_{m \geq 1} \sum\limits_{(\gamma_1,\dots,\gamma_m) \in C_{A}^m \cap C_{B}^m} \varphi(\gamma_1,\dots,\gamma_m) \, \rho^{-2m}\bigg) \\
& = \P_{K,\rho}(A \subset \Mcal)\, \P_{K,\rho}(B \subset \Mcal)  \, \exp\bigg(\sum\limits_{m \geq 1} \, \rho^{-2m} \, \sum\limits_{(\gamma_1,\dots,\gamma_m) \in C_{A}^m \cap C_{B}^m} \varphi(\gamma_1,\dots,\gamma_m) \bigg).
\end{aligned}
\end{equation}
From \eqref{eq:twopoint} we obtain that
\begin{equation} \label{eq:twopoint2}
\begin{aligned}
| \rho^{|A| +|B|} \; U_{K,\rho}(A,B) | 
& \leq  \P_{K,\rho}(A \subset \Mcal)  \; \P_{K,\rho}(B \subset \Mcal) \; \big| e^{\alpha(A,B,\rho)} -1 \big| \\
& \leq | \alpha(A,B,\rho) | \, \max\{1,e^{\alpha(A,B,\rho)}\},
\end{aligned}
\end{equation}
where
$$
\alpha(A,B,\rho):=\sum\limits_{m \geq 1} \, \rho^{-2m} \, \sum\limits_{(\gamma_1,\dots,\gamma_m) \in C_{A}^m \cap C_{B}^m} \varphi(\gamma_1,\dots,\gamma_m).
$$
The last inequality in \eqref{eq:twopoint2} is the mean value inequality applied to the function $x \mapsto e^x$.
Now observe that for any $(\gamma_1,\dots,\gamma_m) \in C_{A}^m \cap C_{B}^m$, $\varphi(\gamma_1,\dots,\gamma_m) \neq 0$ only if the graph $G$, which is obtained from $(\gamma_1,\dots,\gamma_m)$ by drawing an edge between $i$ and $j$ whenever $\zeta(\gamma_i,\gamma_j) \neq 0$, is connected. Stated differently, $\varphi(\gamma_1,\dots,\gamma_m) \neq 0$ only if there exists at least one path connecting a vertex of the set $A$ to a vertex of the set $B$. In particular, it is necessary that $m \geq d(A,B)$. Thus, 
% Warum in Friedli Velenik noch R^d? In meinen Augen kann man es auch dort weglassen.
\begin{equation}
\begin{aligned} \label{eq:application}
 & \sum\limits_{m \geq 1} \rho^{-2m} \, \sum\limits_{(\gamma_1,\dots,\gamma_m) \in C_{A}^m \cap C_{B}^m} |\varphi(\gamma_1,\dots,\gamma_m)| \\
%& \leq \sum\limits_{m \geq d_L(A,B)} e^{-2m} \sum\limits_{(\gamma_1,\dots,\gamma_m) \in C_{e}^m} |\varphi(\gamma_1,\dots,\gamma_m)|  \prod_{i=1}^m \Big(\frac{\rho}{e}\Big)^{-2} \\
& \leq e^{-2 \, d(A,B)} \, \sum\limits_{m \geq 1}  m \, \sum\limits_{\substack{(\gamma_1,\dots,\gamma_m) \in E_K^m: \\ \gamma_1 \cap A \neq \emptyset}} |\varphi(\gamma_1,\dots,\gamma_m)|  \, \Big(\frac{\rho}{e}\Big)^{-2m} \\
& = e^{-2 \, d(A,B)} \, \Big(\frac{\rho}{e}\Big)^{-2} \, \sum\limits_{\gamma_1 \in C_{A}^1} \bigg(1+\sum\limits_{m \geq 2}  m \, \sum\limits_{(\gamma_2,\dots,\gamma_m) \in E_K^{m-1}} |\varphi(\gamma_1,\dots,\gamma_m)|  \, \Big(\frac{\rho}{e}\Big)^{-2(m-1)} \bigg)\\
& \leq e^{-2 \, d(A,B)} \, e^3 \, \rho^{-2} \, |A| \, 2d
\leq  e^{-2 \, d(A,B)},
\end{aligned}
\end{equation}
where in the last two steps we used \eqref{eq:upperbound} and that $\rho \geq \sqrt{e \, |A| \,(4d-1)} \, e$. 
From \eqref{eq:twopoint2} and \eqref{eq:application}, we thus obtain that

\begin{equation*}
|\rho^{|A|+|B|} \; U_{K,\rho}(A,B) | 
\leq  e^{-2 \, d(A,B) } \, \max\{1, e^{e^{-2 \, d(A,B) }}\} \leq  e^{-2 \, d(A,B) +e^{-2}}.
\end{equation*}
This concludes the proof of the proposition.
\end{proof}

We are now ready to prove Theorem \ref{theorem:maintheorem}. The proof is based on Proposition \ref{prop:exponentialdecaycluster} above and on Theorem \ref{theorem:Guerra} below, which provides an analytic result concerning  superharmonic functions. We refer the reader to \cite[Appendix]{Guerra} for definitions and further background.  
For any $a>0$, we denote by  $E_a \subset \C$  the interior of the ellipse with center $a+1$, foci at $a$ and $a+2$ and semi major axis $a+1$, i.e.,
\begin{equation} \label{eq:definitionellipse}
E_a := \Big\{(x,y) \in \C : \frac{(x-(a+1))^2}{(a+1)^2}+\frac{y^2}{(a+1)^2-1} < 1 \Big\}.
\end{equation}
We set $\Omega_a:= E_a \setminus [a,a+2]$. We further set $o:=(0,0) \in \C$.

\begin{theorem}[{\cite[Theorem A.6]{Guerra}}] \label{theorem:Guerra}
Let $a>0$. Suppose that $G: \overline{\Omega_a} \setminus \{o\} \mapsto \R \cup \{\infty\}$ is superharmonic, $G(z) \geq 0$ for all $z \in \overline{\Omega_a} \setminus \{o\}$ and that $G(z) \geq b$ for $a \leq z \leq a+2$. Then 
\begin{equation}
G(z) \geq \frac{b}{\ln(2a+2)} \, \frac{z}{a+1}
\end{equation}
for $0 < z < a$.
\end{theorem}
We remark that in the statement of \cite[Theorem A.6]{Guerra} it is required that $G$ is superharmonic and non-negative on the whole set $\mathbb{H}^+:=\{z \in \mathbb{C} \, | \, Re(z)>0\}$. However, it follows immediately from the proof of the theorem that it suffices to define $G$ on $\overline{\Omega_a} \setminus \{o\}$.

%\textcolor{blue}{Let $\mathbb{H}^+:=\{z \in \mathbb{C} \, | \, Re(z)>0\}$.}

\begin{proof}[Proof of Theorem \ref{theorem:maintheorem}] 
Let $d \geq 1$ and fix two non-empty 
disjoint subsets $A, B \subset \V$. %\textcolor{blue}{and assume without loss of generality that $|A| \leq |B|$}. 
To begin, we note that by \cite[Theorem 10]{Gruber}  the function
$
U_\rho(A,B)
$ 
is an analytic function of $\rho \in \mathbb{H}^+$. Further for any $\rho \in \mathbb{H}^+$, it satisfies the upper bound
\begin{equation} \label{eq:upperbound1}
\begin{aligned}
\big| U_\rho(A,B) \big | 
)
 \leq 2 \, \bigg(\frac{1}{\text{Re}(\rho)}\bigg)^{|A|+|B|} .
\end{aligned}
\end{equation}
The inequality \eqref{eq:upperbound1} is derived in the proofs of \cite[Theorem 9, Theorem 10]{Gruber} and is based on a recurrence relation for the partition function of the monomer-dimer model. 
Let $\rho \in \overline{\Omega_a} \setminus \{o\}$ with $a:= \sqrt{e \, \min\{|A|,|B|\} \, (4d-1)} \, e$. Then by definition \eqref{eq:definitionellipse},
\begin{equation} \label{eq:ReIm}
\text{Im}(\rho)^2 \leq 2 \, (a+1) \, \text{Re}(\rho)-\text{Re}(\rho)^2
\end{equation}
and from \eqref{eq:upperbound1} and \eqref{eq:ReIm} it follows that
\begin{equation} \label{eq:boundabsolutevalue}
\big| \rho^{2 \,(|A|+|B|)} \; U_\rho(A,B) \big | 
\leq 2 \, \bigg(\frac{|\rho|^2}{\text{Re}(\rho)}\bigg)^{|A|+|B|} \\
\leq 2 \, (2a+2)^{|A|+|B|}.
\end{equation}
We introduce the function $G_{A,B}: \overline{\Omega_a} \setminus \{o\} \to \R$ defined by 
$$
G_{A,B}(\rho):= -\ln\bigg|\frac{1}{2} \, \Big(\frac{\rho^{2}}{2a+2}\Big)^{|A|+|B|} \; U_\rho(A,B)\bigg|.
$$
Since $U_\rho(A,B)$ is analytic, the function $G_{A,B}(\rho)$ is superharmonic \cite[Theorem A.3]{Guerra}. From \eqref{eq:boundabsolutevalue} it follows that $G_{A,B}(\rho) \geq 0$ for any $\rho \in \overline{\Omega_a} \setminus \{o\}$.
Further, by Proposition \ref{prop:exponentialdecaycluster} we have that for $a \leq \rho \leq a+2$,
$$
G_{A,B}(\rho)  
= -\ln\bigg|\frac{1}{2} \, \Big(\frac{\rho}{2a+2}\Big)^{|A|+|B|} \; \rho^{|A|+|B|} \, U_\rho(A,B)\bigg|\geq 2  d(A,B).
$$
Applying Theorem \ref{theorem:Guerra} to the function $ G_{A,B}(\rho) $ concludes the proof of the theorem.

\end{proof}

\section*{Acknowledgments}
The author thanks the German Research Foundation through the international research training group 2544 and through the priority program SPP2265 (project number 444084038) for financial support.

\end{document}